\documentclass{svjour3}                     
\smartqed  
\usepackage{setspace}
\usepackage{algorithm}
\usepackage{algorithmic}

\usepackage{hyperref}
\usepackage{enumerate}
\usepackage{amssymb}
\usepackage{booktabs}
\usepackage{graphicx,fancyhdr}
\usepackage{graphics,color}
\usepackage{subfigure}
\usepackage{booktabs}
\usepackage{graphicx}
\usepackage{amsmath}
\usepackage{bm}
\numberwithin{equation}{section}
\numberwithin{theorem}{section}
\numberwithin{lemma}{section}
\numberwithin{remark}{section}

\usepackage{mathptmx}
\begin{document}

\title{Numerical approximations for the fractional Fokker-Planck equation with two-scale diffusion}


\author{Jing Sun$^{1}$, Weihua Deng$^{*,1}$,
Daxin Nie$^{1}$}


\institute{
$^{*}$Corresponding author. E-mail: dengwh@lzu.edu.cn            \\
$^{1}$School of Mathematics and Statistics, Gansu Key Laboratory of Applied Mathematics and Complex Systems, Lanzhou University, Lanzhou 730000, P.R. China \\
}

\date{Received: date / Accepted: date}

\maketitle

\begin{abstract}
Fractional Fokker-Planck equation plays an important role in describing anomalous dynamics. To the best of our knowledge, the existing discussions mainly focus on this kind of equation involving one diffusion operator. In this paper, we first derive the fractional Fokker-Planck equation with two-scale diffusion from the L\'evy process framework, and then the fully discrete scheme is built by using the $L_{1}$ scheme for time discretization and finite element method for space. With the help of the sharp regularity estimate of the solution, we optimally get the spatial and temporal error estimates. Finally, we validate the effectiveness of the provided algorithm by extensive numerical experiments.

\end{abstract}


\section{Introduction}
We provide the numerical methods for the fractional Fokker-Planck equation with two-scale diffusion, i.e.,
\begin{equation}\label{eqretosol}
	\left\{\begin{aligned}
		&{}_{0}\partial^{\alpha}_{t}(u-u_{0})+(-\Delta)u+(-\Delta)^{s}u=f,\quad (x,t)\in\Omega\times(0,T],\\
		&u=0,\quad (x,t)\in\Omega^{c}\times(0,T],\\
		&u(0)=u_{0},\quad x\in\Omega,\\
	\end{aligned}\right .
\end{equation}
where $\Omega\subset \mathbb{R}^{n}$ ($n=1,2,3$) is a bounded domain with smooth boundary and $\Omega^{c}$ means its complement set; $f$ is a given source term; $T$ is the fixed terminal time; $\Delta$ denotes Laplace operator; $(-\Delta)^{s}$ is the fractional Laplacian  defined by
\begin{equation*}
	(-\Delta)^{s}u(x)=c_{n,s}{\rm P.V.}\int_{\mathbb{R}^{n}}\frac{u(x)-u(y)}{|x-y|^{n+2s}}dy,\quad s\in(0,1)
\end{equation*}
with $c_{n,s}=\frac{2^{2s}s\Gamma(n/2+s)}{\pi^{n/2}\Gamma(1-s)}$ and ${\rm P.V.}$ denotes the principal value integral; ${}_{0}\partial^{\alpha}_{t}$ is the Riemann-Liouville fractional derivative defined by \cite{Podlubny.1999Fde}
\begin{equation*}
	{}_{0}\partial^{\alpha}_{t}u=\frac{1}{\Gamma(1-\alpha)}\frac{\partial}{\partial t}\int_{0}^{t}(t-\xi)^{-\alpha}u(\xi)d\xi
\end{equation*}
with $\alpha\in(0,1)$.

Now we briefly state how to derive Eq. \eqref{eqretosol} from the framework of the  L\'evy process \cite{Applebaum.2009Lpasc}.
As we all know, L\'evy process is one of the important stochastic processes with stationary and independent increments and it is thought to be an efficient model to approximate the non-Gaussian process \cite{Applebaum.2009Lpasc}.  Let $\mathbf{L}(t)$ be a L\'evy process with Fourier exponent $\eta(\mathbf{w})$ and take
\begin{equation*}
	\eta(\mathbf{w})=-\mathbf{w}^{T}\mathbf{w}+\int_{\mathbb{R}^{n}\backslash\{0\}}(e^{\mathbf{i}\mathbf{w}^{T}\mathbf{y}}-1-\mathbf{i}\mathbf{w}^{T}\mathbf{y})\nu(d\mathbf{y})
\end{equation*}
with $\mathbf{w}^{T}$ being the transpose of $\mathbf{w}$ and $\nu$  a sigma-finite L\'evy measure on $\mathbb{R}^{n}$, i.e.,
\begin{equation*}
	\nu(d\mathbf{y})=c_{n,s}|\mathbf{y}|^{-1-2s}d\mathbf{y}
\end{equation*}
and $c_{n,s}=\frac{2^{2s}s\Gamma(n/2+s)}{\pi^{n/2}\Gamma(1-s)}$.
Further let $\mathbf{T}(t)$ be a strictly increasing subordinator independent of $\mathbf{L}(t)$ with Laplace exponent $\varPsi(z)=z^{\alpha}$, where $\alpha\in(0,1)$. Define the inverse subordinator $\mathbf{E}(t)=\inf\{r>0: \mathbf{T}(r)>t\}$, which has the probability density function in Laplace space as 
\begin{equation*}
	\tilde{p}_{\mathbf{E}}(r,z)=z^{\alpha-1}e^{-rz^{\alpha}}.
\end{equation*}
 Denote $p_{\mathbf{X}}(x,t)$ as the probability density function  of the stochastic process $\mathbf{X}(t)=\mathbf{L}(\mathbf{E}(t))$.
Introduce $p_{\mathbf{L}}(x,t)$ as the probability density function of $\mathbf{L}(t)$. Using  the following formula
\begin{equation*}
	p_{\mathbf{X}}(x,t)=\int_{0}^{\infty}p_{\mathbf{L}}(x,r)p_{\mathbf{E}}(r,t)dr,
\end{equation*}
we can get the Fourier-Laplace transforms of $p_{\mathbf{X}}(x,t)$, i.e.
\begin{equation*}
	\begin{aligned}
		\mathcal{F}(\tilde{p}_{\mathbf{X}}(x,t))(\mathbf{w},z)=&\int_{0}^{\infty}e^{-zt}\int_{\mathbb{R}^{n}}e^{i\mathbf{w}^{T}\mathbf{L}(\mathbf{E}(t))}p_{\mathbf{X}}(\mathbf{L}(\mathbf{E}(t)),t)d\mathbf{L}(\mathbf{E}(t))dt\\
		= &\int_{0}^{\infty}e^{-zt}\int_{0}^{\infty}\int_{\mathbb{R}^{n}}e^{i\mathbf{w}^{T}\mathbf{L}(\mathbf{E}(t))}p_{\mathbf{L}}(\mathbf{L}(\mathbf{E}(t)),\mathbf{E}(t))d\mathbf{L}(\mathbf{E}(t))p_{\mathbf{E}}(\mathbf{E}(t),t)d\mathbf{E}(t)dt\\
		=&\int_{0}^{\infty}e^{-zt}\int_{0}^{\infty}e^{\mathbf{E}(t)\eta(\mathbf{w})}p_{\mathbf{E}}(\mathbf{E}(t),t)d\mathbf{E}(t)dt\\
		=&\int_{0}^{\infty}e^{\mathbf{E}(t)\eta(\mathbf{w})}z^{\alpha-1}e^{-\mathbf{E}(t)z^{\alpha}}d\mathbf{E}(t)\\
		=&\frac{z^{\alpha-1}}{z^{\alpha}-\eta(\mathbf{w})},
	\end{aligned}
\end{equation*}
where $\mathcal{F}$ means Fourier transform and `$~\tilde{}~$' denotes Laplace transform.
Thus $p_{\mathbf{X}}(x,t)$ satisfies
\begin{equation}\label{eqgen}
	{}_{0}\partial_{t}^{\alpha}(p_{\mathbf{X}}(x,t)-p_{\mathbf{X}}(x,0))+(-\Delta)p_{\mathbf{X}}(x,t)+(-\Delta)^{s}p_{\mathbf{X}}(x,t)=0.
\end{equation}

Fractional Fokker-Planck equations have attracted many attentions in recent years and the corresponding numerical schemes are extensively proposed \cite{Acosta.2019Feaffep,Acosta.2017AFLERoSaFEA,Bonito.2019NaotifL,Bonito.2016Naofporao,Jin.2015AaotLsftsewnd,Jin.2019NmftfeewndAco,Lin.2007Fdsafttfde,Nie.202116NafstsfdedbfGn,Nie.2020NaftstfFPswtis,Yan.2018AAotMLSfTFPDEwND}. But these numerical methods are mainly constructed for the fractional Fokker-Planck equation with one diffusion operator. To the best of our knowledge, the relative numerical discussions are few for this kind of equation involving two-scale diffusion. In this paper, we discuss the numerical method for Eq. \eqref{eqretosol}. To be specific, we first provide a sharp regularity estimate for Eq. \eqref{eqretosol}, in which we treat $(-\Delta)^{s}u$ as a ``source term'' to overcome the difficulties caused by diffusion operators with different scales in regularity analysis; with the help of the elliptic regularity of $(-\Delta)$ and the equivalence of different fractional Sobolev norms, we show that $u\in \hat{H}^{2\min(1,\frac{7}{4}-s-\epsilon)}(\Omega)$ when $u_{0},~f(0)\in L^{2}(\Omega)$ and $\int_{0}^{t}\|f'(r)\|_{L^{2}(\Omega)}dr<\infty$; the detailed proofs can refer to Theorem \ref{Thmreg}. Then the finite element method is used to discretize spatial operator and an auxiliary equation (see \eqref{eqhelpsemi0}) is introduced to help us derive the spatial error estimates. At the same time, the $L_{1}$ scheme \cite{Lin.2007Fdsafttfde} is used to approximate the temporal derivative and an $\mathcal{O}(\tau)$ convergence is obtained.

The rest of the paper is organized as follows. In Section 2, we give some notations and function spaces. The fully discrete scheme is built based on $L_{1}$ discretization in time and finite element methods in space in Section 3. In Section 4,  we first provide the regularity estimate for the solution, and then present the complete error estimates for spatial semi-discrete scheme and the fully discrete scheme, respectively. Some numerical examples, in Section 5, are proposed to support the theory. We conclude the paper with some discussions in the last section. Throughout the paper, $C$ is a generic positive constant, whose value may differ at different places and $\epsilon>0$ is arbitrarily small.

\section{Preliminaries}
\subsection{Notations}
Let $A=(-\Delta)$ and $\mathcal{A}^{s}=(-\Delta)^{s}$ with zero Dirichlet boundary conditions and $\{(\lambda_{i},\phi_{i})\}_{i=1}^{\infty}$ be the eigenvalues and eigenfunctions of $A$, where $\{\phi_{i}\}_{i=1}^{\infty}$ are orthonormal bases in $L^{2}(\Omega)$ and $\lambda_{i}\leq \lambda_{i+1}$ for $i\geq 1$. Introduce $A^{\sigma}$ for $\sigma\in\mathbb{R}$ as
\begin{equation*}
	A^{\sigma}u=\sum_{i=1}^{\infty}\lambda_{i}^{\sigma}(u,\phi_{i})\phi_{i}.
\end{equation*}
Let $\hat{H}^{2\sigma}(\Omega)=\mathcal{D}(A^{\sigma})$ equipped with
$\|u\|_{\hat{H}^{2\sigma}(\Omega)}=\|A^{\sigma}u\|_{L^{2}(\Omega)}$ \cite{Thomee.2006Gfemfpp}, where $\mathcal{D}(A^{\sigma})$ means the domain of $A^{\sigma}$. It is easy to verify that $\hat{H}^{0}(\Omega)=L^{2}(\Omega)$, $\hat{H}^{1}(\Omega)=H^{1}_{0}(\Omega)$, and $\hat{H}^{2}(\Omega)=H^{1}_{0}(\Omega)\cap H^{2}(\Omega)$. In the following, we denote  $\|\cdot\|$ as the operator norm from $L^2(\Omega)$ to $L^2(\Omega)$ and the notation `~$\tilde{}$~' as Laplace transform.

 For $\kappa>0$ and $\pi/2<\theta<\pi$, the definitions of sectors $\Sigma_{\theta}$ and $\Sigma_{\theta,\kappa}$ in the complex plane $\mathbb{C}$ are
\begin{equation*}
	\begin{aligned}
		&\Sigma_{\theta}=\{z\in\mathbb{C}\setminus \{0\}:|\arg z|\leq \theta\}, \\
		&\Sigma_{\theta,\kappa}=\{z\in\mathbb{C}:|z|\geq\kappa,|\arg z|\leq \theta\},\\
	\end{aligned}
\end{equation*}
and the contour $\Gamma_{\theta,\kappa}$ is defined by
\begin{equation*}
	\Gamma_{\theta,\kappa}=\{z\in\mathbb{C}: |z|=\kappa,|\arg z|\leq \theta\}\cup\{z\in\mathbb{C}: z=r e^{\pm \mathbf{i}\theta}: r\geq \kappa\},
\end{equation*}
oriented with an increasing imaginary part, where  $\mathbf{i}^2=-1$.
\subsection{Function spaces}

Following \cite{Acosta.2019Feaffep,Acosta.2017AFLERoSaFEA,Bonito.2019NaotifL,Bonito.2016Naofporao,DiNezza.2012HgttfSs}, for $s\in (0,1)$ and $\Omega\subset \mathbb{R}^{n}$, the fractional Sobolev space can be defined by
\begin{equation*}
	H^{s}(\Omega)=\left \{w\in L^{2}(\Omega):~|w|_{H^{s}(\Omega)}^{2}=\int_{\Omega}\int_{\Omega}\frac{|w(x)-w(y)|^{2}}{|x-y|^{n+2s}}dxdy<\infty\right \}
\end{equation*}
and the corresponding norm is $\|\cdot\|_{H^{s}(\Omega)}=\|\cdot\|_{L^{2}(\Omega)}+|\cdot|_{H^{s}(\Omega)}$; for $s>1$ and $s\notin\mathbb{N}$, we define the fractional Sobolev space $H^{s}(\Omega)$ as
\begin{equation*}
	H^{s}(\Omega)=\left \{w\in H^{\lfloor s\rfloor}(\Omega):~|D^{\alpha}w|_{H^{\sigma}(\Omega)}<\infty~{\rm for~all~}\alpha,~{\rm s.t.}~|\alpha|=\lfloor s\rfloor\right \}
\end{equation*}
with $\sigma=s-\lfloor s\rfloor$ and $\lfloor s\rfloor$ being the biggest integer not larger than $s$.

 On the other hand, a crucial subspace of $H^{s}(\mathbb{R}^{n})$ with $s\in(0,2)$ is defined by
\begin{equation*}
	H^{s}_{0}(\Omega)=\left \{w\in H^{s}(\mathbb{R}^{n}): {\bf supp}~w\subset \bar{\Omega}\right \}
\end{equation*}
and its norm  $\|\cdot\|_{H^{s}_{0}(\Omega)}=\|\cdot\|_{H^{s}(\mathbb{R}^{n})}$; in particular, for $s\in(0,1)$, its norm can also be defined by
\begin{equation*}
	\|u\|_{H^{s}_{0}(\Omega)}=\|u\|_{L^{2}(\Omega)}+\langle u,u\rangle_{s},
\end{equation*}
where
\begin{equation*}
	\langle u,v\rangle_{s}=\frac{c_{n,s}}{2}\int\int_{\mathbb{R}^{n}\times \mathbb{R}^{n}\backslash(\Omega^{c}\times\Omega^{c})}\frac{(u(x)-u(y))(v(x)-v(y))}{|x-y|^{n+2s}}dxdy.
\end{equation*}
Moreover, denote $H^{-s}(\Omega)$ as the dual space of $H^{s}_{0}(\Omega)$ with $s>0$.
\begin{remark}\label{Reeqspace}
	According to \cite{Bonito.2019NaotifL}, when $s\in[0,\frac{3}{2})$ and $\Omega$ is a Lipschitz domain, it holds $\hat{H}^{s}(\Omega)=H^{s}_{0}(\Omega)$;  when $s\in[0,\frac{1}{2})$, there holds $H^{s}(\Omega)=H^{s}_{0}(\Omega)$. Thus, when $s\in(-\frac{3}{2},0]$ and $\Omega$ is a Lipschitz domain,  we have $\hat{H}^{s}(\Omega)=H^{s}(\Omega)$.
\end{remark}

\section{Numerical discretizations}
In this section, we develop the fully-discrete scheme for \eqref{eqretosol} based on $L_{1}$ discretization in time and finite element approximation in space. Denote $\mathcal{T}_{h}$ as a shape regular quasi-unform partition of the domain $\Omega$ with mesh size $h$ and $X_{h}$ as the space of continuous piecewise linear  functions on $\mathcal{T}_{h}$. Let $(\cdot,\cdot)$ be the $L^{2}$ inner product and $P_{h}:~L^{2}(\Omega)\rightarrow X_{h}$ the $L^{2}$ projection defined by
\begin{equation*}
	(P_{h}u,v_{h})=(u,v_{h})\quad \forall v_{h}\in X_{h}.
\end{equation*}
Introduce $R_{h}^{s}:~H^{s}_{0}(\Omega)\rightarrow X_{h}$ with $s\in(0,1)$ satisfying
\begin{equation*}
	\langle u,v_{h}\rangle_{s}=\langle R_{h}^{s}u,v_{h}\rangle_{s}\quad \forall v_{h}\in X_{h}.
\end{equation*}
Using the finite element approximation for the operators $A$ and $\mathcal{A}^{s}$ in Eq. \eqref{eqretosol},  then the semi-discrete scheme can be written as: find $u_{h}\in X_{h}$ satisfying
\begin{equation}\label{eqsemi0}
	({}_{0}\partial^{\alpha}_{t}(u_{h}-u_{0,h}),v_{h})+(\nabla u_{h},\nabla v_{h})+\langle u_{h},v_{h}\rangle_{s}=(f,v_{h})\quad \forall v_{h}\in X_{h},
\end{equation}
where $u_{0,h}=P_{h}u_{0}$.

Introduce two discrete operators $A_{h}$ and $\mathcal{A}^{s}_{h}$ as
\begin{equation*}
	(A_{h}u_{h},v_{h})=(\nabla u_{h},\nabla v_{h}),\quad (\mathcal{A}^{s}_{h}u_{h},v_{h})=\langle u_{h}, v_{h}\rangle_{s} \quad \forall u_{h},v_{h}\in X_{h}.
\end{equation*}
It is easy to verify that \cite{Nie.202116NafstsfdedbfGn}
\begin{equation}\label{eqARPA}
	\mathcal{A}^{s}_{h}R^{s}_{h}=P_{h}\mathcal{A}^{s}.
\end{equation}
Then we can rewrite Eq. \eqref{eqsemi0} as
\begin{equation}\label{eqsemi}
	\left\{\begin{aligned}
		&{}_{0}\partial^{\alpha}_{t}(u_{h}-u_{0,h})+A_{h}u_{h}+\mathcal{A}^{s}_{h}u_{h}=f_{h},\quad (x,t)\in\Omega\times(0,T],\\
		&u_{0,h}=P_{h}u_{0},\qquad x\in\Omega,\\
	\end{aligned}\right .
\end{equation}
where $f_{h}=P_{h}f$.

Next, we use  $L_{1}$ scheme introduced in \cite{Lin.2007Fdsafttfde} to discretize the temporal derivative.
Introduce $b^{(\alpha)}_{j}$ as
\begin{equation*}
	b^{(\alpha)}_{j}=((j+1)^{1-\alpha}-j^{1-\alpha})/\Gamma(2-\alpha),\qquad j=0,1,\cdots,n-1,
\end{equation*}
and ${}_{0}\partial^{\alpha}_{t}(u(t_{n})-u_{0})$ with $\alpha\in(0,1)$ can be approximated by
\begin{equation*}
	{}_{0}\partial^{\alpha}_{t} (u(t_{n})-u_{0})\approx\frac{1}{\tau^{\alpha}}\left(b^{(\alpha)}_{0}\left (u(t_{n})-u_{0}\right )+\sum_{j=1}^{n-1}(b^{(\alpha)}_{j}-b^{(\alpha)}_{j-1})\left(u(t_{n-j})-u_{0}\right)\right).
\end{equation*}
Let
\begin{equation}\label{eqdefdja}
	d^{(\alpha)}_{j}=\tau^{-\alpha}\left\{
	\begin{aligned}
		&b^{(\alpha)}_{0},\qquad j=0,\\
		&b^{(\alpha)}_{j}-b^{(\alpha)}_{j-1},\qquad j>0.
	\end{aligned}\right.
\end{equation}
Then we have
\begin{equation*}
	{}_{0}\partial^{\alpha}_{t} (u(t_{n})-u_{0})\approx\sum_{j=0}^{n-1}d^{(\alpha)}_{j}\left(u(t_{n-j})-u_{0}\right).
\end{equation*}
So, the fully discrete scheme of Eq. \eqref{eqretosol} can be written as
\begin{equation}\label{equfullsche0}
	\left\{
	\begin{aligned}
		&\sum_{j=0}^{n-1}d^{(\alpha)}_{j}\left(u^{n-j}_{h}-u^{0}_{h}\right)+A_{h}u^{n}_{h}+\mathcal{A}_{h}^{s}u^{n}_{h}=f^{n}_{h},\\
		&u^{0}_{h}=u_{0,h},
	\end{aligned}
	\right.
\end{equation}
where $f^{n}_{h}=f_{h}(t_{n})$.
\section{Error analyses}
Here we first provide the regularity of the solution of Eq. \eqref{eqretosol}, and then we develop the error estimates of the spatial semi-discrete scheme and fully-discrete scheme, respectively.

Below we recall the Gr{\"o}nwall inequality.
\begin{lemma}[\cite{Elliott.1992EEwSaNDfaFEMftCHE,Nie.2020NaftstfFPswtis}]\label{lemgrondwall}
	Let the function $\phi(t)\geq 0$ be continuous for $0< t\leq T$. If
	\begin{equation*}
		\phi(t)\leq \sum_{k=1}^Na_kt^{-1+\alpha_k}+b\int_0^t(t-r)^{-1+\beta}\phi(r)dr,\quad 0<t\leq T,
	\end{equation*}
	for some positive constants $\{a_k\}_{k=1}^N$, $\{\alpha_k\}_{k=1}^N$, $b$, and  $\beta$, then there exists a positive constant {$C=C(b,T,\{\alpha_k\}_{k=1}^N,\beta)$ }such that
	\begin{equation*}
		\phi(t)\leq C\sum_{k=1}^Na_kt^{-1+\alpha_k} ~~{\rm for } ~~  0<t\leq T.
	\end{equation*}
\end{lemma}

\subsection{Regularity of the solution}
Taking Laplace transform for \eqref{eqretosol}, one has
\begin{equation*}
	z^{\alpha}\tilde{u}+A\tilde{u}+\mathcal{A}^{s}\tilde{u}=z^{\alpha-1}u_{0}+\tilde{f},
\end{equation*}
resulting in
\begin{equation}\label{eqsolrep0}
	\begin{aligned}
		\tilde{u}=&-\tilde{E}(z)\mathcal{A}^{s}\tilde{u}+\tilde{E}(z)z^{\alpha-1}u_{0}+\tilde{E}(z)\tilde{f},
	\end{aligned}
\end{equation}
where
\begin{equation*}
	\tilde{E}(z)=(z^{\alpha}+A)^{-1}.
\end{equation*}
\begin{theorem}\label{Thmreg}
	Let $u$ be the solution of \eqref{eqretosol}. Assuming $u_{0}\in L^{2}(\Omega)$, $f(0)\in L^{2}(\Omega)$, and $\int_{0}^{t}\|f'(r)\|_{L^{2}(\Omega)}dr<\infty$, we have
	\begin{equation*}
		\|A^{\sigma}u\|_{L^{2}(\Omega)}\leq Ct^{-\sigma\alpha}\|u_{0}\|_{L^{2}(\Omega)}+C\|f(0)\|_{L^{2}(\Omega)}+\int_{0}^{t}\|f'(r)\|_{L^{2}(\Omega)}dr,
	\end{equation*}
	where $\sigma\in(\max(-\epsilon,s-\frac{3}{4}),\min(1,\frac{7}{4}-s-\epsilon)]$.
\end{theorem}
\begin{proof}
	Using \eqref{eqsolrep0} and $f(t)=f(0)+\int_{0}^{t}f'(r)dr$, we split $\|A^{\sigma}u\|_{L^{2}(\Omega)}$ into four parts, i.e.,
	\begin{equation*}
		\begin{aligned}
			&\|A^{\sigma}u\|_{L^{2}(\Omega)}\\
			\leq& C\left \|\int_{\Gamma_{\theta,\kappa}}e^{zt}A^{\sigma}\tilde{E}(z)\mathcal{A}^{s}\tilde{u}dz\right \|_{L^{2}(\Omega)}+C\left \|\int_{\Gamma_{\theta,\kappa}}e^{zt}A^{\sigma}\tilde{E}(z)z^{\alpha-1}u_{0}dz\right \|_{L^{2}(\Omega)}\\
			&+C\left \|\int_{\Gamma_{\theta,\kappa}}e^{zt}A^{\sigma}\tilde{E}(z)z^{-1}f(0)dz\right \|_{L^{2}(\Omega)}+C\left \|\int_{0}^{t}\int_{\Gamma_{\theta,\kappa}}e^{z(t-r)}A^{\sigma}\tilde{E}(z)z^{-1}dzf'(r)dr\right \|_{L^{2}(\Omega)}.
		\end{aligned}
	\end{equation*}
	According to Remark \ref{Reeqspace} and the resolvent estimate $\|(z^{\alpha}+A)^{-1}\|\leq C|z|^{-\alpha}$ \cite{Jin.2015AaotLsftsewnd,Jin.2019NmftfeewndAco,Lubich.1996Ndeefaoaeewaptmt}, there holds
	\begin{equation*}
		\begin{aligned}		
			\|A^{\sigma}u\|_{L^{2}(\Omega)}\leq& C \int_{0}^{t}\int_{\Gamma_{\theta,\kappa}}|e^{z(t-r)}|\|A^{\sigma-\gamma}\tilde{E}(z)\||dz|\|A^{\gamma}\mathcal{A}^{s}u\|_{L^{2}(\Omega)}dr \\
			&+C\int_{\Gamma_{\theta,\kappa}}|e^{zt}||z|^{\sigma \alpha-1}|dz| \|u_{0} \|_{L^{2}(\Omega)}+C\int_{\Gamma_{\theta,\kappa}}|e^{zt}||z|^{(\sigma-1 )\alpha-1}|dz|\|f(0)\|_{L^{2}(\Omega)}\\
			&+C\int_{0}^{t}\int_{\Gamma_{\theta,\kappa}}|e^{z(t-r)}||z|^{(\sigma-1) \alpha-1}|dz|\|f'(r)\|_{L^{2}(\Omega)}dr\\
			\leq& C\int_{0}^{t}(t-r)^{(1-\sigma+\gamma)\alpha-1}\|u(r)\|_{\hat{H}^{2s+2\gamma}(\Omega)}dr\\
			&+Ct^{-\sigma\alpha}\|u_{0}\|_{L^{2}(\Omega)}+Ct^{(1-\sigma) \alpha}\|f(0)\|_{L^{2}(\Omega)}\\
			&+C\int_{0}^{t}(t-r)^{(1-\sigma) \alpha}\|f'(r)\|_{L^{2}(\Omega)}dr,
		\end{aligned}
	\end{equation*}
	where $\gamma\in(\max(\sigma-1,-\frac{3}{4}),\min(\frac{3}{4}-s,\sigma-s,\frac{1}{4}))$ and $\sigma\in(\max(-\epsilon,s-\frac{3}{4}),\min(1,\frac{7}{4}-s-\epsilon)]$. Combining Lemma \ref{lemgrondwall}, the desired results are obtained.
\end{proof}

\subsection{Spatial error estimate}
To get the spatial error estimate,  we first provide the expression of $u_{h}$ from \eqref{eqsemi}.  By Laplace transform, there exists 
\begin{equation}\label{eqsemisolLp0}
	\tilde{u}_{h}=-\tilde{E}_{h}(z)\mathcal{A}^{s}_{h}\tilde{u}_{h}+z^{\alpha-1}\tilde{E}_{h}(z)P_{h}u_{0}+\tilde{E}_{h}(z)P_{h}\tilde{f},
\end{equation}
where
\begin{equation*}
	\tilde{E}_{h}(z)=(z^{\alpha}+A_{h})^{-1}.
\end{equation*}
First, we provide some useful lemmas.
\begin{lemma}[\cite{Bazhlekova.2015AaotRSpfagsgf,Thomee.2006Gfemfpp}]\label{lemeroper0}
	Let $v\in L^2(\Omega)$ and $z\in\Sigma_{\theta,\kappa}$. Denote $w=(z^{\alpha}+A)^{-1}v$ and $w_h=(z^{\alpha}+A_h)^{-1} P_hv$. Then one has
	\begin{equation*}
		\|w-w_h\|_{L^2(\Omega)}+h\|w-w_h\|_{\hat{H}^1(\Omega)}\leq Ch^{2}\|v\|_{L^2(\Omega)}.
	\end{equation*}
\end{lemma}
Similar to the proofs in \cite{Bazhlekova.2015AaotRSpfagsgf}, we have
\begin{lemma}\label{lemeroper1}
	Let $v\in \hat{H}^{-{\sigma}}(\Omega)$ with ${\sigma}\in[0,1]$ and $z\in\Sigma_{\theta,\kappa}$. Denote $w=(z^{\alpha}+A)^{-1}v$ and $w_h=(z^{\alpha}+A_h)^{-1} P_hv$. Then there holds
	\begin{equation*}
		\|w-w_h\|_{L^2(\Omega)}+h\|w-w_h\|_{\hat{H}^1(\Omega)}\leq Ch^{2-{\sigma}}\|v\|_{\hat{H}^{-{\sigma}}(\Omega)}.
	\end{equation*}
\end{lemma}

By using the boundedness of $\|A^{-s}_{h}P_{h}A^{s}\|$ with $s\in[0,\frac{1}{2}]$ \cite{Jin.2015AaotLsftsewnd,Thomee.2006Gfemfpp} and the stability of $L^{2}$ projection $P_{h}$ \cite{Thomee.2006Gfemfpp}, it holds
\begin{equation*}
	\begin{aligned}
		\|\tilde{E}(z)-\tilde{E}_{h}(z)P_{h}\|\leq\left \{
		\begin{aligned}
			&Ch^{2},\\
			&C|z|^{-\alpha},
		\end{aligned}\right .\\
	\end{aligned}
\end{equation*}
and
\begin{equation*}
	\begin{aligned}
		&\|\tilde{E}(z)-\tilde{E}_{h}(z)P_{h}\|_{L^{2}(\Omega)\rightarrow \hat{H}^{1}(\Omega)}\leq Ch;\\
		&\|\tilde{E}(z)-\tilde{E}_{h}(z)P_{h}\|_{\hat{H}^{-1}(\Omega)\rightarrow L^{2}(\Omega)}\leq\left \{
		\begin{aligned}
			&Ch,\\
			&C|z|^{-\alpha/2}.
		\end{aligned}\right .\\
	\end{aligned}
\end{equation*}

In what follows, introduce $\bar{u}_{h}\in X_{h}$ as the solution of the following auxiliary equation to help to obtain the spatial error estimate, i.e.,
\begin{equation}\label{eqhelpsemi0}
	({}_{0}\partial^{\alpha}_{t}(\bar{u}_{h}-u_{0,h}),v_{h})+(\nabla \bar{u}_{h},\nabla v_{h})+\langle u,v_{h}\rangle_{s}=(f,v_{h})\quad \forall v_{h}\in X_{h},
\end{equation}
where $u$ is the exact solution of \eqref{eqretosol} and $u_{0,h}=P_{h}u_{0}$. Equation \eqref{eqhelpsemi0} can be rewritten as
\begin{equation}\label{eqhelpsemi}
	\left\{\begin{aligned}
		&{}_{0}\partial^{\alpha}_{t}(\bar{u}_{h}-u_{0,h})+A_{h}\bar{u}_{h}+P_{h}\mathcal{A}^{s}u=f_{h}\quad (x,t)\in\Omega\times(0,T],\\
		&u_{0,h}=P_{h}u_{0}\quad x\in\Omega.\\
	\end{aligned}\right .
\end{equation}
Taking Laplace transform for \eqref{eqhelpsemi} gives
\begin{equation}\label{eqhelpsolrep}
	\tilde{\bar{u}}_{h}=-\tilde{E}_{h}(z)P_{h}\mathcal{A}^{s}\tilde{u}+z^{\alpha-1}\tilde{E}_{h}(z)P_{h}u_{0}+\tilde{E}_{h}(z)P_{h}\tilde{f}.
\end{equation}
\begin{lemma}\label{lemest0}
	Let $u$ and $\bar{u}_{h}$ be the solutions of \eqref{eqretosol} and \eqref{eqhelpsemi}, respectively. Assume $u_{0}\in L^{2}(\Omega)$, $f(0)\in L^{2}(\Omega)$, and $\int_{0}^{t}\|f'(r)\|_{L^{2}(\Omega)}dr<\infty$. Then the following estimates hold
	\begin{equation*}
		\begin{aligned}
		&\|u-\bar{u}_{h}\|_{\hat{H}^{\sigma}(\Omega)}\\
		&\quad\leq\left \{
		\begin{aligned}
			&Ch^{2-\sigma}t^{-\alpha}\|u_{0}\|_{L^{2}(\Omega)}+Ch^{2-\sigma}\|f(0)\|_{L^{2}(\Omega)}\\
			&\qquad\qquad\qquad\qquad+Ch^{2-\sigma}\int_{0}^{t}\|f'(r)\|_{L^{2}(\Omega)}dr,\quad s\in(0,\frac{3}{4}),\\
			&Ch^{3.5-2s-\epsilon-\sigma}t^{-\alpha}\|u_{0}\|_{L^{2}(\Omega)}+Ch^{3.5-2s-\epsilon-\sigma}\|f(0)\|_{L^{2}(\Omega)}\\
			&\qquad\qquad\qquad\qquad+Ch^{3.5-2s-\epsilon-\sigma}\int_{0}^{t}\|f'(r)\|_{L^{2}(\Omega)}dr,\quad s\in[\frac{3}{4},1)\\
		\end{aligned}\right .
	\end{aligned}
	\end{equation*}
with $\sigma\in[0,1]$.
\end{lemma}

\begin{proof}
For $s\in(0,\frac{3}{4})$, one can get
\begin{equation*}
	\begin{aligned}
		&\|u-\bar{u}_{h}\|_{L^{2}(\Omega)}\\
		\leq
		&C\left \|\int_{\Gamma_{\theta,\kappa}}e^{zt}(\tilde{E}(z)-\tilde{E}_{h}(z)P_{h})\mathcal{A}^{s}\tilde{u}dz\right \|_{L^2(\Omega)} +C\left \|\int_{\Gamma_{\theta,\kappa}}e^{zt}z^{\alpha-1}(\tilde{E}(z)-\tilde{E}_{h}(z)P_{h})u_{0}dz\right \|_{L^{2}(\Omega)}\\
		& +C\left \|\int_{\Gamma_{\theta,\kappa}}e^{zt}z^{-1}(\tilde{E}(z)-\tilde{E}_{h}(z)P_{h})f(0)dz\right \|_{L^{2}(\Omega)}+C\left \|\int_{\Gamma_{\theta,\kappa}}e^{zt}z^{-1}(\tilde{E}(z)-\tilde{E}_{h}(z)P_{h})\tilde{f'}dz\right \|_{L^{2}(\Omega)}.
	\end{aligned}
\end{equation*}
Using \eqref{eqsolrep0}, Lemma \ref{lemeroper1}, and Theorem \ref{Thmreg}, there holds
\begin{equation*}
	\begin{aligned}
		&\left \|\int_{\Gamma_{\theta,\kappa}}e^{zt}(\tilde{E}(z)-\tilde{E}_{h}(z)P_{h})\mathcal{A}^{s}\tilde{u}dz\right \|_{L^2(\Omega)}\\
		\leq&C\left \|\int_{\Gamma_{\theta,\kappa}}e^{zt}(\tilde{E}(z)-\tilde{E}_{h}(z)P_{h})\mathcal{A}^{s}\tilde{E}(z)\mathcal{A}^{s}\tilde{u}dz\right \|_{L^2(\Omega)}\\
		&+C\left \|\int_{\Gamma_{\theta,\kappa}}e^{zt}(\tilde{E}(z)-\tilde{E}_{h}(z)P_{h})\mathcal{A}^{s}\tilde{E}(z)z^{\alpha-1}u_{0}dz\right \|_{L^2(\Omega)}\\
		&+C\left \|\int_{\Gamma_{\theta,\kappa}}e^{zt}(\tilde{E}(z)-\tilde{E}_{h}(z)P_{h})\mathcal{A}^{s}\tilde{E}(z)z^{-1}f(0)dz\right \|_{L^2(\Omega)}\\
		&+C\left \|\int_{\Gamma_{\theta,\kappa}}e^{zt}(\tilde{E}(z)-\tilde{E}_{h}(z)P_{h})\mathcal{A}^{s}\tilde{E}(z)z^{-1}\tilde{f}'dz\right \|_{L^2(\Omega)}\\
		\leq&Ch^{2}t^{-s\alpha}\|u_{0}\|_{L^{2}(\Omega)}+Ch^{2}\|f(0)\|_{L^{2}(\Omega)}\\
		&+C\int_{0}^{t}\int_{\Gamma_{\theta,\kappa}}|e^{z(t-r)}|\|(\tilde{E}(z)-\tilde{E}_{h}(z)P_{h})\|\|\mathcal{A}^{s}\tilde{E}(z)\||dz|\|\mathcal{A}^{s}u(r)\|_{L^{2}(\Omega)}dr\\
		&+C\int_{0}^{t}\int_{\Gamma_{\theta,\kappa}}|e^{z(t-r)}|\|(\tilde{E}(z)-\tilde{E}_{h}(z)P_{h})\|\|\mathcal{A}^{s}\tilde{E}(z)z^{-1}\||dz|\|f'(r)\|_{L^{2}(\Omega)}dr\\
		\leq&Ch^{2}t^{-s\alpha}\|u_{0}\|_{L^{2}(\Omega)}+Ch^{2}\|f(0)\|_{L^{2}(\Omega)}+Ch^{2}\int_{0}^{t}\|f'(r)\|_{L^{2}(\Omega)}dr,
	\end{aligned}
\end{equation*}
which leads to
\begin{equation*}
	\|u-\bar{u}_{h}\|_{L^{2}(\Omega)}\leq Ch^{2}t^{-\alpha}\|u_{0}\|_{L^{2}(\Omega)}+Ch^{2}\|f(0)\|_{L^{2}(\Omega)}+Ch^{2}\int_{0}^{t}\|f'(r)\|_{L^{2}(\Omega)}dr.
\end{equation*}
Similarly, it holds
\begin{equation*}
	\begin{aligned}
		\|u-\bar{u}_{h}\|_{\hat{H}^{1}(\Omega)}\leq Cht^{-\alpha}\|u_{0}\|_{L^{2}(\Omega)}+Ch\|f(0)\|_{L^{2}(\Omega)}+Ch\int_{0}^{t}\|f'(r)\|_{L^{2}(\Omega)}dr.
	\end{aligned}
\end{equation*}
Combining interpolation property \cite{Adams.2003Ss}, we obtain
\begin{equation*}
	\|u-\bar{u}_{h}\|_{\hat{H}^{\sigma}(\Omega)}\leq Ch^{2-\sigma}t^{-\alpha}\|u_{0}\|_{L^{2}(\Omega)}+Ch^{2-\sigma}\|f(0)\|_{L^{2}(\Omega)}+Ch^{2-\sigma}\int_{0}^{t}\|f'(r)\|_{L^{2}(\Omega)}dr
\end{equation*}
with $\sigma\in[0,1]$. Similarly, for $s\in[\frac{3}{4},1)$, there are
\begin{equation*}
	\begin{aligned}
		\|u-\bar{u}_{h}\|_{L^{2}(\Omega)}\leq& Ch^{3.5-2s-\epsilon}t^{-\alpha}\|u_{0}\|_{L^{2}(\Omega)}\\
		&+Ch^{3.5-2s-\epsilon}\|f(0)\|_{L^{2}(\Omega)}+Ch^{3.5-2s-\epsilon}\int_{0}^{t}\|f'(r)\|_{L^{2}(\Omega)}dr,\\
		\|u-\bar{u}_{h}\|_{\hat{H}^{1}(\Omega)}\leq& Ch^{2.5-2s-\epsilon}t^{-\alpha}\|u_{0}\|_{L^{2}(\Omega)}\\
		&+Ch^{2.5-2s-\epsilon}\|f(0)\|_{L^{2}(\Omega)}+Ch^{2.5-2s-\epsilon}\int_{0}^{t}\|f'(r)\|_{L^{2}(\Omega)}dr,
	\end{aligned}
\end{equation*}
and
\begin{equation*}
	\begin{aligned}
		\|u-\bar{u}_{h}\|_{\hat{H}^{
				\sigma}(\Omega)}\leq& Ch^{3.5-2s-\epsilon-\sigma}t^{-\alpha}\|u_{0}\|_{L^{2}(\Omega)}\\
		&+Ch^{3.5-2s-\epsilon-\sigma}\|f(0)\|_{L^{2}(\Omega)}+Ch^{3.5-2s-\epsilon-\sigma}\int_{0}^{t}\|f'(r)\|_{L^{2}(\Omega)}dr
	\end{aligned}
\end{equation*}
with $\sigma\in[0,1]$. The proof is completed. 
\end{proof}

Below we consider the estimate of $\bar{u}_{h}-u_{h}$ in some specific space $\mathbb{V}$ $($$\mathbb{V}$ is $L^{2}(\Omega)$ or $\hat{H}^{2s-1}(\Omega)$$)$. According to Eqs. \eqref{eqsemisolLp0} and \eqref{eqhelpsolrep}, one can split it into the following two parts
\begin{equation}\label{eqdef12}
	\begin{aligned}
		\|\bar{u}_{h}-u_{h}\|_{\mathbb{V}}\leq& C\left\|\int_{\Gamma_{\theta,\kappa}}e^{zt}(\tilde{E}_{h}(z)P_{h}\mathcal{A}^{s}\tilde{u}-\tilde{E}_{h}(z)\mathcal{A}^{s}_{h}\tilde{u}_{h})dz\right \|_{\mathbb{V}}\\
		\leq& C\left\|\int_{\Gamma_{\theta,\kappa}}e^{zt}(\tilde{E}_{h}(z)P_{h}\mathcal{A}^{s}\tilde{u}-\tilde{E}_{h}(z)\mathcal{A}^{s}_{h}R^{s}_{h}\tilde{u}_{h})dz\right \|_{\mathbb{V}}\\
		\leq& C\left\|\int_{\Gamma_{\theta,\kappa}}e^{zt}\tilde{E}_{h}(z)P_{h}\mathcal{A}^{s}(\tilde{u}-\tilde{u}_{h})dz\right \|_{\mathbb{V}}\\
		\leq& C\left\|\int_{\Gamma_{\theta,\kappa}}e^{zt}(\tilde{E}_{h}(z)P_{h}-\tilde{E}(z))\mathcal{A}^{s}(\tilde{u}-\tilde{u}_{h})dz\right \|_{\mathbb{V}}\\
		&+ C\left\|\int_{\Gamma_{\theta,\kappa}}e^{zt}\tilde{E}(z)\mathcal{A}^{s}(\tilde{u}-\tilde{u}_{h})dz\right \|_{\mathbb{V}}\\
		\leq &C\|\uppercase\expandafter{\romannumeral1}\|_{\mathbb{V}}+C\|\uppercase\expandafter{\romannumeral2}\|_{\mathbb{V}},
	\end{aligned}
\end{equation}
where we use the definition of $R_{h}^{s}$ and \eqref{eqARPA}.

 Next, we consider the estimates of $\uppercase\expandafter{\romannumeral1}$ and $\uppercase\expandafter{\romannumeral2}$ in different spaces.

\begin{lemma}\label{lemestp1}
	Let $\uppercase\expandafter{\romannumeral1}$ be defined in \eqref{eqdef12}, $u_{0}\in L^{2}(\Omega)$, $f(0)\in L^{2}(\Omega)$, and $\int_{0}^{t}\|f'(r)\|_{L^{2}(\Omega)}dr<\infty$. Then the following estimates hold
	\begin{equation*}
		\|\uppercase\expandafter{\romannumeral1}\|_{L^{2}(\Omega)}\leq\left \{
		\begin{aligned}
			&Ch^{(2-2s)(1-\epsilon)}\int_{0}^{t}(t-r)^{(1-s)\alpha\epsilon-1}\left\|u(r)-u_{h}(r)\right \|_{L^{2}(\Omega)}dr,\quad s\in(0,\frac{1}{2}],\\
			&Ch^{1-\epsilon}\int_{0}^{t}(t-r)^{\frac{\alpha\epsilon}{2}-1}\left\|u(r)-u_{h}(r)\right \|_{\hat{H}^{2s-1}(\Omega)}dr,\quad s\in(\frac{1}{2},1),
		\end{aligned}\right .
	\end{equation*}
and
\begin{equation*}
	\|\uppercase\expandafter{\romannumeral1}\|_{\hat{H}^{2s-1}(\Omega)}\leq C\int_{0}^{t}(t-r)^{(1-s)\alpha-1}\left\|u(r)-u_{h}(r)\right \|_{\hat{H}^{2s-1}(\Omega)}dr,\quad s\in(\frac{1}{2},1).
\end{equation*}
\end{lemma}
\begin{proof}
	For $s\in(0,\frac{1}{2}]$, Lemma \ref{lemeroper1}, Remark \ref{Reeqspace}, and the fact $\mathcal{A}^{s}: H^{l}(\mathbb{R}^{n})\rightarrow H^{l-2s}(\mathbb{R}^{n})$ give
\begin{equation*}
	\begin{aligned}
		\|\uppercase\expandafter{\romannumeral1}\|_{L^{2}(\Omega)}\leq& C\left\|\int_{\Gamma_{\theta,\kappa}}e^{zt}(\tilde{E}_{h}(z)P_{h}-\tilde{E}(z))A^{s}A^{-s}\mathcal{A}^{s}(\tilde{u}-\tilde{u}_{h})dz\right \|_{L^{2}(\Omega)}\\
		\leq& C\int_{0}^{t}\int_{\Gamma_{\theta,\kappa}}|e^{z(t-r)}|\|(\tilde{E}_{h}(z)P_{h}-\tilde{E}(z))A^{s}\||dz|\\
		&\qquad\cdot\left\|A^{-s}\mathcal{A}^{s}(u(r)-u_{h}(r))\right \|_{L^{2}(\Omega)}dr\\
		\leq& Ch^{(2-2s)(1-\epsilon)}\int_{0}^{t}(t-r)^{(1-s)\alpha\epsilon-1}\left\|u(r)-u_{h}(r)\right \|_{L^{2}(\Omega)}dr.
	\end{aligned}
\end{equation*}

Similarly, for $s\in(\frac{1}{2},1)$, one can get
\begin{equation*}
	\begin{aligned}
		\|\uppercase\expandafter{\romannumeral1}\|_{L^{2}(\Omega)}\leq& C\left\|\int_{\Gamma_{\theta,\kappa}}e^{zt}(\tilde{E}_{h}(z)P_{h}-\tilde{E}(z))A^{\frac{1}{2}}A^{-\frac{1}{2}}\mathcal{A}^{s}(\tilde{u}-\tilde{u}_{h})dz\right \|_{L^{2}(\Omega)}\\
		\leq& C\int_{0}^{t}\int_{\Gamma_{\theta,\kappa}}|e^{z(t-r)}|\|(\tilde{E}_{h}(z)P_{h}-\tilde{E}(z))A^{\frac{1}{2}}\||dz|\\
		&\qquad\cdot\left\|A^{-\frac{1}{2}}\mathcal{A}^{s}(u(r)-u_{h}(r))\right \|_{L^{2}(\Omega)}dr\\
		\leq& Ch^{1-\epsilon}\int_{0}^{t}(t-r)^{\frac{\alpha\epsilon}{2}-1}\left\|u(r)-u_{h}(r)\right \|_{\hat{H}^{2s-1}(\Omega)}dr.
	\end{aligned}
\end{equation*}

On the other hand, for $s\in(\frac{1}{2},1)$, there holds
\begin{equation*}
	\begin{aligned}
		\|\uppercase\expandafter{\romannumeral1}\|_{\hat{H}^{2s-1}(\Omega)}\leq& C\left\|\int_{\Gamma_{\theta,\kappa}}e^{zt}A^{s-\frac{1}{2}}(\tilde{E}_{h}(z)P_{h}-\tilde{E}(z))A^{\frac{1}{2}}A^{-\frac{1}{2}}\mathcal{A}^{s}(\tilde{u}-\tilde{u}_{h})dz\right \|_{L^{2}(\Omega)}\\
		\leq& C\int_{0}^{t}\int_{\Gamma_{\theta,\kappa}}|e^{z(t-r)}|\|A^{s-\frac{1}{2}}(\tilde{E}_{h}(z)P_{h}-\tilde{E}(z))A^{\frac{1}{2}}\||dz|\\
		&\qquad\cdot\left\|A^{-\frac{1}{2}}\mathcal{A}^{s}(u(r)-u_{h}(r))\right \|_{L^{2}(\Omega)}dr\\
		\leq& C\int_{0}^{t}(t-r)^{(1-s)\alpha-1}\left\|u(r)-u_{h}(r)\right \|_{\hat{H}^{2s-1}(\Omega)}dr,
	\end{aligned}
\end{equation*}
where we have used Lemma \ref{lemeroper1}, Remark \ref{Reeqspace}, and the fact $\mathcal{A}^{s}: H^{l}(\mathbb{R}^{n})\rightarrow H^{l-2s}(\mathbb{R}^{n})$.
\end{proof}
\begin{lemma}\label{lemestp2}
	Let $\uppercase\expandafter{\romannumeral2}$ be defined in \eqref{eqdef12}. If $u_{0}\in L^{2}(\Omega)$, $f(0)\in L^{2}(\Omega)$, and $\int_{0}^{t}\|f'(r)\|_{L^{2}(\Omega)}dr<\infty$, then we have
	\begin{equation*}
		\|\uppercase\expandafter{\romannumeral2}\|_{L^{2}(\Omega)}\leq
		C\int_{0}^{t}(t-r)^{(1-s)\alpha-1}\|u(r)-u_{h}(r)\|_{L^{2}(\Omega)}dr,\quad s\in(0,\frac{3}{4})
	\end{equation*}
and
\begin{equation*}
	\|\uppercase\expandafter{\romannumeral2}\|_{\hat{H}^{2s-1}(\Omega)}\leq C\int_{0}^{t}(t-r)^{(1-s)\alpha-1}\|u(r)-u_{h}(r)\|_{\hat{H}^{2s-1}(\Omega)}dr,\quad s\in(\frac{1}{2},1).
\end{equation*}
\end{lemma}
\begin{proof}
For $s\in(0,\frac{3}{4})$, using resolvent estimate \cite{Jin.2015AaotLsftsewnd,Jin.2019NmftfeewndAco,Lubich.1996Ndeefaoaeewaptmt} and Remark \ref{Reeqspace}, one has
\begin{equation*}
	\begin{aligned}
		\|\uppercase\expandafter{\romannumeral2}\|_{L^{2}(\Omega)}\leq& C\left\|\int_{\Gamma_{\theta,\kappa}}e^{zt}\tilde{E}(z)A^{s}A^{-s}\mathcal{A}^{s}(\tilde{u}-\tilde{u}_{h})dz\right \|_{L^{2}(\Omega)}\\
		\leq& C\int_{0}^{t}\int_{\Gamma_{\theta,\kappa}}|e^{z(t-r)}|\|\tilde{E}(z)A^{s}\||dz|\left\|A^{-s}\mathcal{A}^{s}(u(r)-u_{h}(r))\right \|_{L^{2}(\Omega)}dr\\
		\leq& C\int_{0}^{t}(t-r)^{(1-s)\alpha-1}\|u(r)-u_{h}(r)\|_{L^{2}(\Omega)}dr.
	\end{aligned}
\end{equation*}
Moreover, for $s\in(\frac{1}{2},1)$, simple calculations lead to
\begin{equation*}
	\begin{aligned}
		\|\uppercase\expandafter{\romannumeral2}\|_{\hat{H}^{2s-1}(\Omega)}\leq& C\left\|\int_{\Gamma_{\theta,\kappa}}e^{zt}A^{s-\frac{1}{2}}\tilde{E}(z)A^{\frac{1}{2}}A^{-\frac{1}{2}}\mathcal{A}^{s}(\tilde{u}-\tilde{u}_{h})dz\right \|_{L^{2}(\Omega)}\\
		\leq& C\int_{0}^{t}(t-r)^{(1-s)\alpha-1}\|u(r)-u_{h}(r)\|_{\hat{H}^{2s-1}(\Omega)}dr.
	\end{aligned}
\end{equation*}
The proof is completed. 
\end{proof}

Thanks to the above lemmas, we can provide the following error estimate for the spatial semi-discrete scheme.

 \begin{theorem}\label{thmspa}
	Let $u$ and $u_{h}$ be the solutions of Eqs. \eqref{eqretosol} and \eqref{eqsemi}, respectively. Assuming $u_{0}\in L^{2}(\Omega)$, $f(0)\in L^{2}(\Omega)$, and $\int_{0}^{t}\|f'(r)\|_{L^{2}(\Omega)}dr<\infty$, we have
	\begin{equation*}
		\begin{aligned}
			&\|u-u_{h}\|_{L^{2}(\Omega)}\\
			&\quad\leq
			Ch^{2}t^{-\alpha}\|u_{0}\|_{L^{2}(\Omega)}+Ch^{2}\|f(0)\|_{L^{2}(\Omega)}+Ch^{2}\int_{0}^{t}\|f'(r)\|_{L^{2}(\Omega)}dr, \quad s\in(0,\frac{3}{4}),\\
		\end{aligned}
	\end{equation*}
	and
	\begin{equation*}
		\begin{aligned}
			&\|u-u_{h}\|_{\hat{H}^{2s-1}(\Omega)}\\
			&\quad\leq \left\{
			\begin{aligned}
				&Ct^{-\alpha}h^{3-2s}\|u_{0}\|_{L^{2}(\Omega)}+Ch^{3-2s}\|f(0)\|_{L^{2}(\Omega)}\\
				&\qquad\qquad\qquad+Ch^{3-2s}\int_{0}^{t}\|f'(r)\|_{L^{2}(\Omega)}dr, \quad s\in(\frac{1}{2},\frac{3}{4}),\\
				&Ct^{-\alpha}h^{4.5-4s-\epsilon}\|u_{0}\|_{L^{2}(\Omega)}+Ch^{4.5-4s-\epsilon}\|f(0)\|_{L^{2}(\Omega)}\\
				&\qquad\qquad\qquad+Ch^{4.5-4s-\epsilon}\int_{0}^{t}\|f'(r)\|_{L^{2}(\Omega)}dr, \quad s\in[\frac{3}{4},1).\\
			\end{aligned}
			\right.
		\end{aligned}
	\end{equation*}
\end{theorem}
\begin{proof}
	For $s\in(0,\frac{1}{2}]$, Lemmas \ref{lemest0}, \ref{lemestp1}, and \ref{lemestp2} give
	\begin{equation*}
		\begin{aligned}
			&\|u-u_{h}\|_{L^{2}(\Omega)}\\
			\leq&\|u-\bar{u}_{h}\|_{L^{2}(\Omega)}+\|\bar{u}_{h}-u_{h}\|_{L^{2}(\Omega)}\\
			\leq& Ch^{2}t^{-\alpha}\|u_{0}\|_{L^{2}(\Omega)}+Ch^{2}\|f(0)\|_{L^{2}(\Omega)}+Ch^{2}\int_{0}^{t}\|f'(r)\|_{L^{2}(\Omega)}dr\\
			& +Ch^{(2-2s)(1-\epsilon)}\int_{0}^{t}(t-r)^{(1-s)\alpha\epsilon-1}\left\|u(r)-u_{h}(r)\right \|_{L^{2}(\Omega)}dr\\
			&+C\int_{0}^{t}(t-r)^{(1-s)\alpha-1}\|u(r)-u_{h}(r)\|_{L^{2}(\Omega)}dr\\
			\leq& Ch^{2}t^{-\alpha}\|u_{0}\|_{L^{2}(\Omega)}+Ch^{2}\|f(0)\|_{L^{2}(\Omega)}+Ch^{2}\int_{0}^{t}\|f'(r)\|_{L^{2}(\Omega)}dr\\
			&+C\int_{0}^{t}(t-r)^{(1-s)\alpha\epsilon-1}\left\|u(r)-u_{h}(r)\right \|_{L^{2}(\Omega)}dr.
		\end{aligned}
	\end{equation*}
Thus, by Lemma \ref{lemgrondwall}, one has
\begin{equation*}
	\|u-u_{h}\|_{L^{2}(\Omega)}\leq Ch^{2}t^{-\alpha}\|u_{0}\|_{L^{2}(\Omega)}+Ch^{2}\|f(0)\|_{L^{2}(\Omega)}+Ch^{2}\int_{0}^{t}\|f'(r)\|_{L^{2}(\Omega)}dr.
\end{equation*}
For $s\in(\frac{1}{2},\frac{3}{4})$, according to Lemmas \ref{lemest0}, \ref{lemestp1}, and \ref{lemestp2}, the estimate of $\|u-u_{h}\|_{\hat{H}^{2s-1}(\Omega)}$ can be got, i.e.,
\begin{equation*}
	\begin{aligned}
		&\|u-u_{h}\|_{\hat{H}^{2s-1}(\Omega)}\\
		\leq&\|u-\bar{u}_{h}\|_{\hat{H}^{2s-1}(\Omega)}+\|\bar{u}_{h}-u_{h}\|_{\hat{H}^{2s-1}(\Omega)}\\
		\leq& Ch^{3-2s}t^{-\alpha}\|u_{0}\|_{L^{2}(\Omega)}+Ch^{3-2s}\|f(0)\|_{L^{2}(\Omega)}+Ch^{3-2s}\int_{0}^{t}\|f'(r)\|_{L^{2}(\Omega)}dr\\
		&+C\int_{0}^{t}(t-r)^{(1-s)\alpha-1}\|u(r)-u_{h}(r)\|_{\hat{H}^{2s-1}(\Omega)}dr.
	\end{aligned}
\end{equation*}
Thus one has
\begin{equation*}
	\|u-u_{h}\|_{\hat{H}^{2s-1}(\Omega)}\leq Ch^{3-2s}t^{-\alpha}\|u_{0}\|_{L^{2}(\Omega)}+Ch^{3-2s}\|f(0)\|_{L^{2}(\Omega)}+Ch^{3-2s}\int_{0}^{t}\|f'(r)\|_{L^{2}(\Omega)}dr.
\end{equation*}
Combining the above estimate, we get
\begin{equation*}
	\begin{aligned}
		&\|u-u_{h}\|_{L^{2}(\Omega)}\\
		\leq&\|u-\bar{u}_{h}\|_{L^{2}(\Omega)}+\|\bar{u}_{h}-u_{h}\|_{L^{2}(\Omega)}\\
		\leq& Ct^{-\alpha}h^{2}\|u_{0}\|_{L^{2}(\Omega)}+Ch^{2}\|f(0)\|_{L^{2}(\Omega)}+Ch^{2}\int_{0}^{t}\|f'(r)\|_{L^{2}(\Omega)}dr\\
		& +Ch^{1-\epsilon}\int_{0}^{t}(t-r)^{\frac{\alpha\epsilon}{2}-1}\left\|u(r)-u_{h}(r)\right \|_{\hat{H}^{2s-1}(\Omega)}dr\\
		&+C\int_{0}^{t}(t-r)^{(1-s)\alpha-1}\|u(r)-u_{h}(r)\|_{L^{2}(\Omega)}dr\\
		\leq& Ch^{2}t^{-\alpha}\|u_{0}\|_{L^{2}(\Omega)}+Ch^{2}\|f(0)\|_{L^{2}(\Omega)}+Ch^{2}\int_{0}^{t}\|f'(r)\|_{L^{2}(\Omega)}dr\\
		&+C\int_{0}^{t}(t-r)^{(1-s)\alpha-1}\|u(r)-u_{h}(r)\|_{L^{2}(\Omega)}dr,
	\end{aligned}
\end{equation*}
which leads to the desired result by Lemma \ref{lemgrondwall}.

Similarly, for $s\in(\frac{3}{4},1)$, there exists
\begin{equation*}
	\begin{aligned}
		&\|u-u_{h}\|_{\hat{H}^{2s-1}(\Omega)}\\
		&\qquad\leq Ch^{4.5-4s-\epsilon}t^{-\alpha}\|u_{0}\|_{L^{2}(\Omega)}+Ch^{4.5-4s-\epsilon}\|f(0)\|_{L^{2}(\Omega)}+Ch^{4.5-4s-\epsilon}\int_{0}^{t}\|f'(r)\|_{L^{2}(\Omega)}dr.
	\end{aligned}
\end{equation*}
The proof is completed.
\end{proof}

\subsection{Temporal error estimate}
 Introduce $\mathcal{L}_{h}=A_{h}+\mathcal{A}_{h}^{s}$ and $\mu_{h}^{n}=u_{h}^{n}-u_{h}^{0}$. Then the fully discrete scheme \eqref{equfullsche0} can be rewritten as
\begin{equation}\label{equfullsche}
	\begin{aligned}
		&\sum_{j=0}^{n-1}d^{(\alpha)}_{j}\mu^{n-j}_{h}+\mathcal{L}_{h}\mu^{n}_{h}=f^{n}_{h}-\mathcal{L}_{h}u^{0}_{h}.
	\end{aligned}
\end{equation}
Let $\mu_{h}=u_{h}-u_{h}^{0}$. Then the semi-discrete scheme \eqref{eqsemi} can also be represented as
\begin{equation}\label{equsemi1}
	\begin{aligned}
		&{}_{0}\partial^{\alpha}_{t}\mu_{h}+\mathcal{L}_{h}\mu_{h}=f_{h}-\mathcal{L}_{h}u^{0}_{h}.
	\end{aligned}
\end{equation}
Multiplying $\zeta^{n}$ on both sides of \eqref{equfullsche} and summing $n$ from $1$ to $\infty$ shows
\begin{equation*}
	\sum_{n=1}^{\infty}\sum_{j=0}^{n-1}d^{(\alpha)}_{j}\mu^{n-j}_{h}\zeta^{n}+\sum_{n=1}^{\infty}\mathcal{L}_{h}\mu^{n}_{h}\zeta^{n}=\sum_{n=1}^{\infty}f^{n}_{h}\zeta^{n}-\mathcal{L}_{h}\sum_{n=1}^{\infty}u^{0}_{h}\zeta^{n},
\end{equation*}
which leads to
\begin{equation*}
	\left (\varphi^{(\alpha)}(\zeta)+\mathcal{L}_{h}\right )\sum_{n=1}^{\infty}\mu^{n}_{h}\zeta^{n}=\sum_{n=1}^{\infty}f^{n}_{h}\zeta^{n}-\mathcal{L}_{h}\sum_{n=1}^{\infty}u^{0}_{h}\zeta^{n}
\end{equation*}
with
\begin{equation*}
	\begin{aligned}
		\varphi^{(\alpha)}(\zeta)=&\sum_{j=0}^{\infty}d^{(\alpha)}_{j}\zeta^{j}
		=\frac{\tau^{-\alpha}}{\Gamma(2-\alpha)}\frac{(1-\zeta)^{2}}{\zeta}Li_{\alpha-1}(\zeta).
	\end{aligned}
\end{equation*}
Here, the definitions of $d^{(\alpha)}_{j}$ can refer to \eqref{eqdefdja} and $Li_p(z)$ is defined as \cite{Lewin.1981Paaf}
\begin{equation*}
	Li_p(z)=\sum_{j=1}^{\infty}\frac{z^j}{j^p}.
\end{equation*}
Thus, we have
\begin{equation*}
	\sum_{n=1}^{\infty}\mu^{n}_{h}\zeta^{n}=\left (\varphi^{(\alpha)}(\zeta)+\mathcal{L}_{h}\right )^{-1}\left (\sum_{j=1}^{\infty}f^{j}_{h}\zeta^{j}-\mathcal{L}_{h}\sum_{j=1}^{\infty}u^{0}_{h}\zeta^{j}\right ),
\end{equation*}
which leads to
\begin{equation*}
	\mu^{n}_{h}=\frac{1}{2\pi\mathbf{i}}\int_{|\zeta|=\xi_{\tau}}\zeta^{-n-1}\left (\varphi^{(\alpha)}(\zeta)+\mathcal{L}_{h}\right )^{-1}\left (\sum_{j=1}^{\infty}f^{j}_{h}\zeta^{j}-\mathcal{L}_{h}\sum_{j=1}^{\infty}u^{0}_{h}\zeta^{j}\right )d\zeta
\end{equation*}
with $\xi_{\tau}=e^{-\tau(\kappa+1)}$. Taking $\zeta=e^{-z\tau}$, one obtains
\begin{equation*}
	\mu^{n}_{h}=\frac{\tau}{2\pi\mathbf{i}}\int_{\Gamma^{\tau}}e^{zt_{n}}\left (\varphi^{(\alpha)}(e^{-z\tau})+\mathcal{L}_{h}\right )^{-1}\left (\sum_{j=1}^{\infty}f^{j}_{h}e^{-zt_{j}}-\mathcal{L}_{h}\sum_{j=1}^{\infty}u^{0}_{h}e^{-zt_{j}}\right )dz,
\end{equation*}
where $\Gamma^{\tau}=\{z=\kappa+1+\mathbf{i}y:~y\in \mathbb{R}~{\rm and}~|y|\leq \pi/\tau\}$. Transforming $\Gamma^{\tau}$ to $\Gamma^{\tau}_{\theta,\kappa}=\{z\in\mathbb{C}:~\kappa\leq |z|\leq \frac{\pi}{\tau\sin(\theta)},~|\arg z|=\theta\}\cup\{z\in\mathbb{ C}:~|z|=\kappa,~|\arg z|\leq \theta\}$, we have
 \begin{equation*}
 	\mu^{n}_{h}=\frac{\tau}{2\pi\mathbf{i}}\int_{\Gamma^{\tau}_{\theta,\kappa}}e^{zt_{n}}\left (\varphi^{(\alpha)}(e^{-z\tau})+\mathcal{L}_{h}\right )^{-1}\left (\sum_{j=1}^{\infty}f^{j}_{h}e^{-zt_{j}}-\mathcal{L}_{h}\sum_{j=1}^{\infty}u^{0}_{h}e^{-zt_{j}}\right )dz.
 \end{equation*}
Since $f(t)=f(0)+\int_{0}^{t}f'(r)dr$ and taking $R_{h}(t)=P_{h}\int_{0}^{t}f'(r)dr$, one has
\begin{equation*}
	\begin{aligned}
		\mu^{n}_{h}=&\frac{\tau}{2\pi\mathbf{i}}\int_{\Gamma_{\theta,\kappa}^{\tau}}e^{zt_{n}}\left (\varphi^{(\alpha)}(e^{-z\tau})+\mathcal{L}_{h}\right )^{-1}\left (\frac{e^{-z\tau}}{1-e^{-z\tau}}f^{0}_{h}-\frac{e^{-z\tau}}{1-e^{-z\tau}}\mathcal{L}_{h}u^{0}_{h}\right )dz\\
		&+\frac{\tau}{2\pi\mathbf{i}}\int_{\Gamma_{\theta,\kappa}^{\tau}}e^{zt_{n}}\left (\varphi^{(\alpha)}(e^{-z\tau})+\mathcal{L}_{h}\right )^{-1}\sum_{j=1}^{\infty}R_{h}(t_{j})e^{-zt_{j}}dz.
	\end{aligned}
\end{equation*}

Similarly, the solution of \eqref{equsemi1} can also be written as
\begin{equation*}
	\begin{aligned}
		\mu_{h}(t_{n})=&\frac{1}{2\pi\mathbf{i}}\int_{\Gamma_{\theta,\kappa}}e^{zt_{n}}\left (z^{\alpha}+\mathcal{L}_{h}\right )^{-1}\left (z^{-1}f^{0}_{h}-z^{-1}\mathcal{L}_{h}u^{0}_{h}\right )dz\\
		&+\frac{1}{2\pi\mathbf{i}}\int_{\Gamma_{\theta,\kappa}}e^{zt_{n}}\left (z^{\alpha}+\mathcal{L}_{h}\right )^{-1}\tilde{R}_{h}dz.\\
	\end{aligned}
\end{equation*}

\begin{lemma}[\cite{Jin.2015AaotLsftsewnd,Yan.2018AAotMLSfTFPDEwND}]
	Let $z\in\Gamma_{\theta,\kappa}^{\tau}$ with $\theta\in(\frac{\pi}{2},\frac{5\pi}{6})$. Then there hold
	\begin{equation*}
		|\varphi^{(\alpha)}(e^{-z\tau})-z^{\alpha}|\leq C\tau^{2-\alpha}|z|^{2}
	\end{equation*}
and
\begin{equation*}
	|\varphi^{(\alpha)}(e^{-z\tau})|\geq C|z|\tau^{1-\alpha}.
\end{equation*}
\end{lemma}
 \begin{theorem}\label{thmtime}
	Assuming $u_{0}\in L^{2}(\Omega)$, $f(0)\in L^{2}(\Omega)$, and $\int_{0}^{t}\|f'(r)\|_{L^{2}(\Omega)}dr<\infty$, we have
	\begin{equation*}
		\begin{aligned}
			\|u_{h}(t_{n})-u^{n}_{h}\|_{L^{2}(\Omega)}\leq&C\tau t_{n}^{-1}\|u^{0}_{h}\|_{L^{2}(\Omega)}+C\tau t_{n}^{-1}\|f^{0}_{h}\|_{L^{2}(\Omega)}\\
			&+C\tau\int_{0}^{t_{n}}(t_{n}-r)^{\alpha-1}\|f'(r)\|_{L^{2}(\Omega)}dr.
		\end{aligned}
	\end{equation*}
\end{theorem}
\begin{proof}
	According to the definitions of $\mu_{h}(t_{n})$ and $\mu^{n}_{h}$, it holds
	\begin{equation*}
		\begin{aligned}
			&\|\mu_{h}(t_{n})-\mu^{n}_{h}\|_{L^{2}(\Omega)}\\
			\leq&C\left\|\int_{\Gamma_{\theta,\kappa}}e^{zt_{n}}\left (z^{\alpha}+\mathcal{L}_{h}\right )^{-1}z^{-1}\mathcal{L}_{h}dz-\int_{\Gamma_{\theta,\kappa}^{\tau}}e^{zt_{n}}\left (\varphi^{(\alpha)}(e^{-z\tau})+\mathcal{L}_{h}\right )^{-1}\frac{\tau e^{-z\tau}}{1-e^{-z\tau}}\mathcal{L}_{h}dz\right\|\|u^{0}_{h}\|_{L^{2}(\Omega)}\\
			&+C\left\|\int_{\Gamma_{\theta,\kappa}}e^{zt_{n}}\left (z^{\alpha}+\mathcal{L}_{h}\right )^{-1}z^{-1}dz-\int_{\Gamma_{\theta,\kappa}^{\tau}}e^{zt_{n}}\left (\varphi^{(\alpha)}(e^{-z\tau})+\mathcal{L}_{h}\right )^{-1}\frac{\tau e^{-z\tau}}{1-e^{-z\tau}}dz\right\|\|f^{0}_{h}\|_{L^{2}(\Omega)}\\
			&+C\left \|\int_{\Gamma_{\theta,\kappa}}e^{zt_{n}}\left (z^{\alpha}+\mathcal{L}_{h}\right )^{-1}\tilde{R}_{h}dz-\int_{\Gamma_{\theta,\kappa}^{\tau}}e^{zt_{n}}\left (\varphi^{(\alpha)}(e^{-z\tau})+\mathcal{L}_{h}\right )^{-1}\tau\sum_{j=1}^{\infty}R_{h}(t_{j})e^{-zt_{j}}dz\right \|_{L^{2}(\Omega)}\\
			\leq&\uppercase\expandafter{\romannumeral1}+\uppercase\expandafter{\romannumeral2}+\uppercase\expandafter{\romannumeral3}.
		\end{aligned}
	\end{equation*}
For $\uppercase\expandafter{\romannumeral1}$, there holds

\begin{equation*}
	\begin{aligned}
		\uppercase\expandafter{\romannumeral1}\leq&C\Bigg(\left \|\int_{\Gamma_{\theta,\kappa}\backslash\Gamma_{\theta,\kappa}^{\tau}}e^{zt_{n}}\left (z^{\alpha}+\mathcal{L}_{h}\right )^{-1}z^{-1}\mathcal{L}_{h}dz\right \|\\
		&+\left \|\int_{\Gamma_{\theta,\kappa}^{\tau}}(e^{zt_{n}}-e^{zt_{n-1}})\left (z^{\alpha}+\mathcal{L}_{h}\right )^{-1}z^{-1}\mathcal{L}_{h}dz\right \|\\
		&+\left \|\int_{\Gamma_{\theta,\kappa}^{\tau}}e^{zt_{n-1}}\left(\left (z^{\alpha}+\mathcal{L}_{h}\right )^{-1}z^{-1}-\left (\varphi^{(\alpha)}(e^{-z\tau})+\mathcal{L}_{h}\right )^{-1}\frac{\tau }{1-e^{-z\tau}}\right) \mathcal{L}_{h}dz\right \|\Bigg)\|u^{0}_{h}\|_{L^{2}(\Omega)}.
	\end{aligned}
\end{equation*}
Using the following fact
\begin{equation*}
	\begin{aligned}
		&\left\|\left (\left (z^{\alpha}+\mathcal{L}_{h}\right )^{-1}z^{-1}-\left (\varphi^{(\alpha)}(e^{-z\tau})+\mathcal{L}_{h}\right )^{-1}\frac{\tau }{1-e^{-z\tau}}\right )\mathcal{L}_{h}\right \|\\
		\leq& \left\|\left (\left (z^{\alpha}+\mathcal{L}_{h}\right )^{-1}z^{-1}-\left (z^{\alpha}+\mathcal{L}_{h}\right )^{-1}\frac{\tau }{1-e^{-z\tau}}\right )\mathcal{L}_{h}\right \|\\
		&+\left\|\left (\left (z^{\alpha}+\mathcal{L}_{h}\right )^{-1}\frac{\tau }{1-e^{-z\tau}}-\left (\varphi^{(\alpha)}(e^{-z\tau})+\mathcal{L}_{h}\right )^{-1}\frac{\tau }{1-e^{-z\tau}}\right )\mathcal{L}_{h}\right \|\\
		\leq& C\tau
	\end{aligned}
\end{equation*}
and the resolvent estimate \cite{Jin.2015AaotLsftsewnd,Jin.2019NmftfeewndAco,Lubich.1996Ndeefaoaeewaptmt}, we have
\begin{equation*}
	\uppercase\expandafter{\romannumeral1}\leq C\tau t_{n}^{-1}\|u^{0}_{h}\|_{L^{2}(\Omega)}.
\end{equation*}
Similarly, one has
\begin{equation*}
	\uppercase\expandafter{\romannumeral2}\leq C\tau t_{n}^{-1}\|f^{0}_{h}\|_{L^{2}(\Omega)}.
\end{equation*}
Doing simple calculations yields
\begin{equation*}
	\begin{aligned}
		\tau\sum_{n=1}^{\infty}R_{h}(t_{n})e^{-zt_{n}}=&\tau\sum_{n=1}^{\infty}\int_{0}^{t_{n}}P_{h}f'(r)dre^{-zt_{n}}\\
		=&\tau\sum_{n=1}^{\infty}\sum_{j=1}^{n}\int_{t_{j-1}}^{t_{j}}P_{h}f'(r)dre^{-zt_{n}}\\
		=&\tau\sum_{j=1}^{\infty}\left (\int_{t_{j-1}}^{t_{j}}P_{h}f'(r)dr\sum_{n=j}^{\infty}e^{-zt_{n}}\right )\\
		=&\frac{\tau}{1-e^{-z\tau}}\sum_{j=1}^{\infty}\left (e^{-zt_{j}}\int_{t_{j-1}}^{t_{j}}P_{h}f'(r)dr\right ),
	\end{aligned}
\end{equation*}
which leads to
\begin{equation*}
	\begin{aligned}
		&\Bigg \|\int_{0}^{t_{n}}\int_{\Gamma_{\theta,\kappa}^{\tau}}e^{z(t_{n}-r)}\left (\varphi^{(\alpha)}(e^{-z\tau})+\mathcal{L}_{h}\right )^{-1}\frac{\tau}{1-e^{-z\tau}}dzP_{h}f'(r)dr\\
		&\qquad\qquad-\int_{\Gamma_{\theta,\kappa}^{\tau}}e^{zt_{n}}\left (\varphi^{(\alpha)}(e^{-z\tau})+\mathcal{L}_{h}\right )^{-1}\tau\sum_{j=1}^{\infty}R_{h}(t_{j})e^{-zt_{j}}dz\Bigg \|_{L^{2}(\Omega)}\\
		\leq &\Bigg \|\int_{0}^{t_{n}}\int_{\Gamma_{\theta,\kappa}^{\tau}}e^{z(t_{n}-r)}\left (\varphi^{(\alpha)}(e^{-z\tau})+\mathcal{L}_{h}\right )^{-1}\frac{\tau}{1-e^{-z\tau}}dzP_{h}f'(r)dr\\
		&\qquad\qquad-\sum_{j=1}^{n}\int_{t_{j-1}}^{t_{j}}\int_{\Gamma_{\theta,\kappa}^{\tau}}e^{z(t_{n}-t_{j})}\left (\varphi^{(\alpha)}(e^{-z\tau})+\mathcal{L}_{h}\right )^{-1}\frac{\tau}{1-e^{-z\tau}}dzP_{h}f'(r)dr\Bigg \|_{L^{2}(\Omega)}\\
		\leq &C\tau\int_{0}^{t_{n}}(t_{n}-r)^{\alpha-1}\|f'(r)\|_{L^{2}(\Omega)}dr.
	\end{aligned}
\end{equation*}
Thus
\begin{equation*}
	\begin{aligned}
		\uppercase\expandafter{\romannumeral3}\leq&\left \|\int_{0}^{t_{n}}\left (\int_{\Gamma_{\theta,\kappa}}e^{z(t_{n}-r)}\left (z^{\alpha}+\mathcal{L}_{h}\right )^{-1}z^{-1}dz-\int_{\Gamma_{\theta,\kappa}^{\tau}}e^{z(t_{n}-r)}\left (\varphi^{(\alpha)}(e^{-z\tau})+\mathcal{L}_{h}\right )^{-1}\frac{\tau}{1-e^{-z\tau}}dz\right )P_{h}f'(r)dr\right \|_{L^{2}(\Omega)}\\
		&+C\tau\int_{0}^{t_{n}}(t_{n}-r)^{\alpha-1}\|f'(r)\|_{L^{2}(\Omega)}dr\\
		\leq&\left \|\int_{0}^{t_{n}}\int_{\Gamma_{\theta,\kappa}\backslash\Gamma_{\theta,\kappa}^{\tau}}e^{z(t_{n}-r)}\left (z^{\alpha}+\mathcal{L}_{h}\right )^{-1}z^{-1}dzP_{h}f'(r)dr\right \|_{L^{2}(\Omega)}\\
		&+\left \|\int_{0}^{t_{n}}\int_{\Gamma_{\theta,\kappa}^{\tau}}e^{z(t_{n}-r)}\left (\left (z^{\alpha}+\mathcal{L}_{h}\right )^{-1}z^{-1}-\left (\varphi^{(\alpha)}(e^{-z\tau})+\mathcal{L}_{h}\right )^{-1}\frac{\tau}{1-e^{-z\tau}}\right )dzP_{h}f'(r)dr\right \|_{L^{2}(\Omega)}\\
		&+C\tau\int_{0}^{t_{n}}(t_{n}-r)^{\alpha-1}\|f'(r)\|_{L^{2}(\Omega)}dr.
	\end{aligned}
\end{equation*}
Similar to the estimation of $\uppercase\expandafter{\romannumeral1}$, we have
\begin{equation*}
	\uppercase\expandafter{\romannumeral3}\leq C\tau\int_{0}^{t_{n}}(t_{n}-r)^{\alpha-1}\|f'(r)\|_{L^{2}(\Omega)}dr.
\end{equation*}
Collecting the above estimates leads to the desired results.
\end{proof}

\section{Numerical experiments}
In this section, some numerical examples are presented to validate our theoretical results. Here, we take $\Omega=(0,1)$ and $T=1$. The following  two initial values and source terms will  be used
\begin{enumerate}[(a)]
	\item\label{conda} \begin{equation*}
		u_{0}(x)=\chi_{(0.5,1)}(x),\quad f(x,t)=0;
	\end{equation*}
\item \label{condb}
\begin{equation*}
	u_{0}(x)=0,\quad f(x,t)=t^{0.1}x^{-0.2},
\end{equation*}
\end{enumerate}
where $\chi_{(a,b)}(x)$ means the characteristic function on $(a,b)$. Due to the exact solution is unknown, we use $e_{h}$ and $e_{\tau}$ to measure spatial and temporal errors, whose definitions are
\begin{equation*}
	e_{h}=u_{h}-u_{h/2},\quad e_{\tau}=u_{\tau}-u_{\tau/2}.
\end{equation*}
Here $u_{h}$ and $u_{\tau}$  denote the numerical solutions under the mesh size $h$ and time stepsize $\tau$, respectively. Thus the resulting convergence rates in some specific space $\mathbb{V}$ can be calculated by
\begin{equation*}
	Rate=\frac{\log(\|e_{h}\|_{\mathbb{V}}/\|e_{h/2}\|_{\mathbb{V}})}{\log(2)}
\end{equation*}
and
\begin{equation*}
	Rate=\frac{\log(\|e_{\tau}\|_{\mathbb{V}}/\|e_{\tau/2}\|_{\mathbb{V}})}{\log(2)}.
\end{equation*}
\begin{example}
We take \eqref{conda} as the initial value and source term to verify the convergence in temporal direction. Here, to avoid the influence of the spatial errors on temporal errors, we take $h=1/512$. The corresponding $L^{2}$ errors and convergence rates are presented in Table \ref{tab:timehom} and all results agree with Theorem \ref{thmtime}.
\begin{table}[htbp]
	\caption{$L^{2}$ errors and convergence rates with the initial value and source term \eqref{conda} }
	\begin{tabular}{ccccccc}
		\hline\noalign{\smallskip}
		$(\alpha,s)\backslash 1/\tau$ & 16 & 32 & 64 & 128 & 256 \\
		\noalign{\smallskip}\hline\noalign{\smallskip}
		(0.4,0.3) & 1.722E-04 & 8.360E-05 & 4.116E-05 & 2.041E-05 & 1.015E-05 \\
		& Rate & 1.0425 & 1.0224 & 1.0122 & 1.0068 \\
		
		(0.4,0.7) & 1.435E-04 & 6.972E-05 & 3.435E-05 & 1.704E-05 & 8.481E-06 \\
		&Rate  & 1.0410 & 1.0214 & 1.0115 & 1.0063 \\
		(0.8,0.3) & 1.843E-04 & 8.519E-05 & 4.074E-05 & 1.979E-05 & 9.695E-06 \\
		&Rate  & 1.1130 & 1.0642 & 1.0415 & 1.0297 \\
		(0.8,0.7) & 1.396E-04 & 6.519E-05 & 3.137E-05 & 1.531E-05 & 7.522E-06 \\
		&Rate  & 1.0982 & 1.0554 & 1.0353 & 1.0248 \\
		\noalign{\smallskip}\hline
	\end{tabular}
	\label{tab:timehom}
\end{table}

\end{example}
\begin{example}
	
Here, we show the $L^{2}$ errors and convergence rates in temporal direction with the initial value and source term \eqref{condb}. We take $h=1/512$ to decrease the influence caused by spatial discretization. All the corresponding results are shown in  Table \ref{tab:timeimhom} and in excellent
agreement with the theoretical predictions.

\begin{table}[htbp]
	\caption{$L^{2}$ errors and convergence rates with the initial value and source term \eqref{condb}}
	\begin{tabular}{ccccccc}
		\hline\noalign{\smallskip}
		$(\alpha,s)\backslash 1/\tau$& 16 & 32 & 64 & 128 & 256 \\
		\noalign{\smallskip}\hline\noalign{\smallskip}
		(0.3,0.4) & 1.845E-05 & 8.321E-06 & 3.805E-06 & 1.754E-06 & 8.124E-07 \\
		&Rate  & 1.1490 & 1.1287 & 1.1172 & 1.1105 \\
		(0.3,0.8) & 1.143E-05 & 5.155E-06 & 2.359E-06 & 1.088E-06 & 5.038E-07 \\
		&Rate  & 1.1482 & 1.1281 & 1.1168 & 1.1102 \\
		(0.7,0.4) & 3.087E-05 & 1.343E-05 & 5.975E-06 & 2.691E-06 & 1.221E-06 \\
		&Rate  & 1.2010 & 1.1682 & 1.1506 & 1.1401 \\
		(0.7,0.8) & 1.774E-05 & 7.751E-06 & 3.459E-06 & 1.562E-06 & 7.099E-07 \\
		&Rate  & 1.1946 & 1.1640 & 1.1474 & 1.1373 \\
		\noalign{\smallskip}\hline
	\end{tabular}
	\label{tab:timeimhom}
\end{table}
\end{example}

\begin{example}
	In this example, we validate the spatial convergence of our scheme with the initial value and source term \eqref{conda}. Here we take $\tau=1/512$ to avoid the influence on errors caused by temporal discretization. Tables \ref{tab:spahoml20} and \ref{tab:spahomH2s1} show the $L^{2}$ errors with $s\in(0,\frac{3}{4})$ and $\hat{H}^{2s-1}$ errors with $s\in(\frac{3}{4},1)$ and all convergence rates agree with the predicted rates in Theorem \ref{thmspa}. Moreover, we provide $L^{2}(\Omega)$ errors and convergence rates when $s\in (\frac{3}{4},1)$ in Table \ref{tab:spahoml21} and the corresponding convergence rates exactly coincide with the Sobolev regularity of the solution.
	\begin{table}[htbp]
		\caption{$L^{2}$ errors and convergence rates with the initial value and source term \eqref{conda} and $s\in(0,\frac{3}{4})$}
		\begin{tabular}{cccccc}
			\hline\noalign{\smallskip}
			$(\alpha,s)\backslash 1/h$ & 16 & 32 & 64 & 128 & 256   \\
				\noalign{\smallskip}\hline\noalign{\smallskip}
			(0.4,0.3) & 1.152E-04 & 2.879E-05 & 7.198E-06 & 1.799E-06 & 4.499E-07  \\
			&Rate  & 2.0002 & 2.0000 & 2.0000 & 1.9998   \\
			(0.6,0.3) & 8.099E-05 & 2.025E-05 & 5.061E-06 & 1.265E-06 & 3.164E-07   \\
			&Rate  & 2.0002 & 2.0000 & 2.0000 & 1.9998   \\
			(0.4,0.7) & 9.964E-05 & 2.531E-05 & 6.427E-06 & 1.631E-06 & 4.133E-07   \\
			&Rate  & 1.9772 & 1.9774 & 1.9786 & 1.9804   \\
			(0.6,0.7) & 6.937E-05 & 1.762E-05 & 4.474E-06 & 1.135E-06 & 2.878E-07   \\
			&Rate  & 1.9773 & 1.9773 & 1.9783 & 1.9801   \\
			\noalign{\smallskip}\hline
		\end{tabular}
		\label{tab:spahoml20}
	\end{table}
	
	\begin{table}[htbp]
		\caption{$\hat{H}^{2s-1}(\Omega)$ errors and convergence rates with the initial value and source term \eqref{conda} and $s\in(\frac{3}{4},1)$}
		\begin{tabular}{cccccc}
			\hline\noalign{\smallskip}
			$(\alpha,s)\backslash 1/h$ & 16 & 32 & 64 & 128 & 256 \\
			\noalign{\smallskip}\hline\noalign{\smallskip}
			(0.3,0.8) & 9.836E-04 & 3.777E-04 & 1.460E-04 & 5.676E-05 & 2.216E-05 \\
			&Rate  & 1.3808 & 1.3711 & 1.3633 & 1.3568 \\
			(0.8,0.8) & 2.950E-04 & 1.133E-04 & 4.387E-05 & 1.707E-05 & 6.677E-06 \\
			&Rate  & 1.3799 & 1.3694 & 1.3613 & 1.3546 \\
			(0.3,0.9) & 2.098E-03 & 9.365E-04 & 4.227E-04 & 1.936E-04 & 9.042E-05 \\
			&Rate  & 1.1640 & 1.1477 & 1.1264 & 1.0984 \\
			(0.3,0.9) & 6.263E-04 & 2.796E-04 & 1.263E-04 & 5.795E-05 & 2.712E-05 \\
			&Rate  & 1.1635 & 1.1464 & 1.1241 & 1.0952 \\
			\noalign{\smallskip}\hline
		\end{tabular}
		\label{tab:spahomH2s1}
	\end{table}

\begin{table}[htbp]
	\caption{$L^{2}$ errors and convergence rates with the initial value and source term \eqref{conda} and $s\in(\frac{3}{4},1)$}
	\begin{tabular}{cccccc}
		\hline\noalign{\smallskip}
		$(\alpha,s)\backslash 1/h$ & 16 & 32 & 64 & 128 & 256 \\
		\noalign{\smallskip} \hline\noalign{\smallskip}
		(0.3,0.8) & 1.046E-04 & 2.704E-05 & 7.049E-06 & 1.853E-06 & 4.908E-07 \\
		&Rate  & 1.9515 & 1.9394 & 1.9277 & 1.9166 \\
		(0.8,0.8) & 3.253E-05 & 8.411E-06 & 2.196E-06 & 5.784E-07 & 1.536E-07 \\
		&Rate  & 1.9513 & 1.9375 & 1.9246 & 1.9126 \\
		(0.3,0.9) & 9.104E-05 & 2.360E-05 & 6.253E-06 & 1.707E-06 & 4.840E-07 \\
		&Rate  & 1.9479 & 1.9160 & 1.8729 & 1.8185 \\
		(0.3,0.9) & 2.813E-05 & 7.290E-06 & 1.935E-06 & 5.300E-07 & 1.511E-07 \\
		&Rate  & 1.9480 & 1.9138 & 1.8680 & 1.8107 \\
		\noalign{\smallskip} \hline
	\end{tabular}
	\label{tab:spahoml21}
\end{table}

\end{example}

\begin{example}
		Lastly, we validate the spatial convergence with the initial value and source term \eqref{condb} and choose $\tau=1/512$. We present the $L^{2}$ errors with $s\in(0,\frac{3}{4})$ and $\hat{H}^{2s-1}$ errors with $s\in(\frac{3}{4},1)$ in Tables \ref{tab:spaimhoml20} and \ref{tab:spaimhomH2s1} and all the convergence rates are consistent with results in Theorem \ref{thmspa}. Moreover, we provide $L^{2}(\Omega)$ errors and convergence rates with $s\in (\frac{3}{4},1)$ in Table \ref{tab:spaimhoml21} and the corresponding convergence rates exactly agree with the Sobolev regularity of the solution.
	\begin{table}[htbp]
		\caption{$L^{2}$ errors and convergence rates with the initial value and source term \eqref{condb} and $s\in(0,\frac{3}{4})$}
		\begin{tabular}{cccccc}
			\hline \noalign{\smallskip}
			$(\alpha,s)\backslash 1/h$ & 16 & 32 & 64 & 128 & 256 \\
			\noalign{\smallskip} \hline \noalign{\smallskip}
			(0.3,0.2) & 2.917E-04 & 7.361E-05 & 1.851E-05 & 4.645E-06 & 1.164E-06 \\
			& Rate & 1.9867 & 1.9916 & 1.9946 & 1.9966 \\
			(0.8,0.2) & 3.037E-04 & 7.657E-05 & 1.925E-05 & 4.828E-06 & 1.210E-06 \\
			& Rate & 1.9878 & 1.9922 & 1.9950 & 1.9968 \\
			(0.3,0.6) & 2.809E-04 & 7.209E-05 & 1.836E-05 & 4.649E-06 & 1.172E-06 \\
			&Rate & 1.9622 & 1.9734 & 1.9815 & 1.9880 \\
			(0.8,0.6) & 2.900E-04 & 7.434E-05 & 1.892E-05 & 4.790E-06 & 1.207E-06 \\
			&Rate & 1.9636 & 1.9742 & 1.9819 & 1.9884 \\
			\noalign{\smallskip} \hline
		\end{tabular}
		\label{tab:spaimhoml20}
	\end{table}

	\begin{table}[htbp]
		\caption{$\hat{H}^{2s-1}$ errors and convergence rates with the initial value and source term \eqref{condb} and $s\in(\frac{3}{4},1)$}
		\begin{tabular}{cccccc}
			\hline\noalign{\smallskip}
			$(\alpha,s)\backslash 1/h$ & 16 & 32 & 64 & 128 & 256 \\
			\noalign{\smallskip} \hline\noalign{\smallskip}
			(0.4,0.8) & 2.387E-03 & 9.462E-04 & 3.751E-04 & 1.488E-04 & 5.906E-05 \\
			& 0 & 1.3348 & 1.3348 & 1.3341 & 1.3329 \\
			(0.6,0.8) & 2.403E-03 & 9.527E-04 & 3.777E-04 & 1.498E-04 & 5.947E-05 \\
			& 0 & 1.3350 & 1.3349 & 1.3341 & 1.3328 \\
			(0.4,0.9) & 5.050E-03 & 2.322E-03 & 1.078E-03 & 5.079E-04 & 2.436E-04 \\
			& 0 & 1.1209 & 1.1066 & 1.0862 & 1.0600 \\
			(0.6,0.9) & 5.083E-03 & 2.337E-03 & 1.085E-03 & 5.112E-04 & 2.452E-04 \\
			& 0 & 1.1211 & 1.1067 & 1.0862 & 1.0599 \\
			\noalign{\smallskip} \hline
		\end{tabular}
		\label{tab:spaimhomH2s1}
	\end{table}

	\begin{table}[htbp]
		\caption{$L^{2}$ errors and convergence rates with the initial value and source term \eqref{condb} and $s\in(\frac{3}{4},1)$}
		\begin{tabular}{cccccc}
			\hline\noalign{\smallskip}
			$(\alpha,s)\backslash 1/h$ & 16 & 32 & 64 & 128 & 256 \\
			\noalign{\smallskip} \hline\noalign{\smallskip}
			(0.4,0.8) & 2.604E-04 & 6.996E-05 & 1.884E-05 & 5.087E-06 & 1.378E-06 \\
			& 0 & 1.8963 & 1.8928 & 1.8888 & 1.8846 \\
			(0.6,0.8) & 2.632E-04 & 7.067E-05 & 1.903E-05 & 5.138E-06 & 1.392E-06 \\
			& 0 & 1.8967 & 1.8930 & 1.8887 & 1.8844 \\
			(0.4,0.9) & 2.244E-04 & 6.064E-05 & 1.672E-05 & 4.739E-06 & 1.387E-06 \\
			& 0 & 1.8879 & 1.8582 & 1.8193 & 1.7722 \\
			(0.6,0.9) & 2.267E-04 & 6.122E-05 & 1.688E-05 & 4.785E-06 & 1.401E-06 \\
			& 0 & 1.8885 & 1.8584 & 1.8191 & 1.7717 \\
			\noalign{\smallskip} \hline
		\end{tabular}
		\label{tab:spaimhoml21}
	\end{table}
	
\end{example}
\section{Conclusions}
In this paper, the fractional Fokker-Planck equation involving two diffusion operators with different scales is derived from the framework of the L\'evy process and this kind of equation can describe the physical phenomena more delicately. We use $L_{1}$ approximation in time and finite element methods in space to get the numerical scheme of the equation. Thanks to the sharp regularity estimate of the solution, we optimally obtain the spatial and temporal error estimates. Extensive numerical experiments validate the theoretical results.

\section*{Acknowledgements}
This work was supported by the National Natural Science Foundation of China under Grant No. 12071195, the AI and Big Data Funds under Grant No. 2019620005000775, and
the Fundamental Research Funds for the Central Universities under Grant No. lzujbky-2021-it26.

\bibliographystyle{spmpsci}
\bibliography{dbo}
\end{document}